\documentclass[12pt,reqno]{amsart}
\usepackage{graphicx, cite, hyperref}
\usepackage{amsfonts}
\usepackage{enumerate}
\usepackage[caption = false]{subfig}
\usepackage{varioref}

\numberwithin{equation}{section}
\DeclareMathOperator{\RE}{Re}
 
 \theoremstyle{plain}

\newtheorem{thm}{Theorem}[section]
\newtheorem{cor}[thm]{Corollary}

\newtheorem{lem}[thm]{Lemma}

\theoremstyle{definition}

\theoremstyle{remark}
\newtheorem{rem}[thm]{Remark}
\makeatother

\begin{document}

\title[Sharp bounds on  Coefficient functionals ......]{ Sharp bounds on  Coefficient functionals 
of certain Sakaguchi functions}

\author[S. Kumar]{Sushil Kumar}
\address{Bharati Vidyapeeth's College of Engineering, Delhi, India, 110 063}
\email{sushilkumar16n@gmail.com}

\author[R. K. pandey]{Rakesh Kumar pandey}
\address{Department of Mathematics, University  of Delhi, Delhi, India, 110007}
\email{rkpandey@maths.du.ac.in}

\author[P. Rai]{Pratima Rai*}
\address{Department of Mathematics, University  of Delhi, Delhi, India, 110007}
\email{ prai@maths.du.ac.in}

\textwidth=455pt \evensidemargin=8pt \oddsidemargin=8pt
\marginparsep=8pt \marginparpush=8pt \textheight=600pt
\topmargin=15pt
\parskip .3cm
\numberwithin{equation}{section}

\begin{abstract}
We determine sharp bounds on some Hankel determinants involving initial coefficients, inverse coefficients, and logarithmic  inverse coefficients for two subclasses of Sakaguchi  functions which are associated with the right half of the lemniscate of Bernoulli and the exponential function. 
Further, we compute sharp bounds on the second Hermitian-Toeplitz determinants involving logarithmic coefficients and logarithmic inverse coefficients. We also discuss invariant property for the obtained estimates with respect to various coefficients.

%
\end{abstract}

\keywords{Sakaguchi  functions; logarithmic inverse
coefficients;  logarithmic coefficients; inverse coefficients;  Hankel determinants and Hermitian-Toeplitz  determinants\\
* Corresponding Author}

\maketitle

\section{Introduction}
Let  $\mathcal{A}$ be   the class  of all normalized analytic functions $f$ of the form $f(z)=z+\sum_{n=2}^{\infty} a_n z^n$  defined on  $\mathbb{D}=\lbrace z\in \mathbb{C} : \vert z\vert <1\rbrace$ and $\mathcal{S}$ be the subclass of $\mathcal{A}$ consisting  of  schlicht functions. There are various  subclasses of the class $\mathcal{S}$ but many researchers paid attention to the classes of starlike and convex functions during 1920-1990 due to proof of Bieberbach conjecture~\cite{Duren}. In 1959, Sakaguchi~\cite{sakaguchi59} introduced the class of starlike functions with respect to symmetric points defined by $\mathcal{S}^*_S:=\{ f\in\mathcal{S}:\, \RE\left( {2zf'(z)}/{((f(z)-f(-z)))}\right) >0\}$; $( z\in\mathbb{D}).$ These functions are also known as Sakaguchi functions which are close-to-convex as well as univalent. Sakaguchi functions are interesting because they combine symmetry properties and starlike properties, making them a unique subclass of univalent functions. These functions deal with the geometric properties of analytic functions. Sakaguchi functions have applications in various areas of complex analysis, such as the theory of univalent functions, quasiconformal mappings, and the study of conformal maps between different regions in the complex plane. Conformal mappings and geometric transformations, related to the properties of starlike functions, can be used in computer graphics and image processing to map shapes, textures, and images in visually pleasing and efficient ways.
In 2004, Ravichandran~\cite{Ravi04} introduced a unified class $\mathcal{S}^*_S(\phi)$ of Sakaguchi functions which is given as
\[\mathcal{S}^*_S(\phi)=\left\{f\in \mathcal{S}: \frac{2zf'(z)}{(f(z)-f(-z))} \prec \phi(z); \,\, z\in \mathbb{D}\right\};\] where 
  $\phi(z)=1+\sum_{n=1}^{\infty}D_nz^n$ is univalent starlike function with respect to 1 which maps $\mathbb{D}$ onto a symmetric region with respect to real axis in the right half plane.   
The class $\mathcal{S}^*_S(\phi)$ consists of many subclasses of $\mathcal{S}^*_S$ which are associated with  bounded region in the right-half of lemniscate of Bernoulli, cosine hyperbolic function,  exponential function etc..  By taking $\phi(z)=e^z$ and $\sqrt{1+z}$, authors~\cite{khatter16} considered subclasses $\mathcal{S}_{S,e}^*$ and $\mathcal{S}_{S,L}^* $ of Sakaguchi  functions and  investigated  bounds  on the initial coefficients of the functions belonging to these classes. Analytically, 
\begin{align*}
 \mathcal{S}_{S,e}^* =\left\{ f\in S: \frac{2zf'(z)}{f(z)-f(-z)}\prec e^z \right \}  \,\,\mbox{and}\,\, \mathcal{S}_{S,L}^* =\left\{ f\in S: \frac{2zf'(z)}{f(z)-f(-z)}\prec \sqrt{1+z} \right \}
\end{align*}
for $z\in \mathbb{D}$. Shanmugam et al.~\cite{Shan06} gave a function
$K_{\phi_n}$ $(n=2,3,....)$ which belong to the class $\mathcal{S}^*_S(\phi)$ such that 
${2z K'_{\phi_n}(z)}/{(K_{\phi_n}(z)-K_{\phi_n}(-z))}=\phi(z^{n-1}),$ $ K_{\phi_n}(0)=K'_{\phi_n}(0)-1=0$. In view of the function $K_{\phi_n} \in \mathcal{S}^*_S(\phi)$, it is noted that the functions $f_1$ and $f_2$ defined as
\begin{equation}\label{eqt1}
f_1(z)=z+\frac{1}{2}z^3+\frac{1}{4}z^5+...\cdot
\end{equation}
and 
\begin{equation}\label{eqt2}
f_2(z)=z+\frac{1}{2}z^2+\frac{1}{4}z^3+\frac{5}{48}z^4+...
\end{equation}
belong to the class $\mathcal{S}_{S,e}^*$. Further, the functions $g_1$ and $g_2$ defined by
\begin{equation}\label{eqt7}
g_1(z)=z+\frac{1}{4}z^3+...\cdot 
\end{equation}
and 
\begin{equation}\label{eqt72}
g_2(z)=z+\frac{1}{4}z^2-\frac{1}{16}z^3+...\cdot
\end{equation}
belong to the class $\mathcal{S}_{S,L}^* $. Thus, it  is observed that the classes $\mathcal{S}_{S,e}^*$ and $\mathcal{S}_{S,L}^* $ are non-empty.

A coefficient functional is a linear functional that operates on the coefficients of a power series representation of an analytic function. By analyzing the behavior of these functionals, one can classify different classes of analytic functions, such as univalent functions, starlike functions, convex functions, and more.
Bounds on coefficients and coefficient functionals such as Fekete-Szeg\"o functional, various Hankel determinants, Hemitian-Toeplitz determinant  provide tools for understanding the intricate relationship between the algebraic properties of power series coefficients and the geometric properties of the corresponding analytic functions. 
If the function $f\in\mathcal{S},$ then the function $F$ which is  inverse of the function $f$ has following expansion
\begin{equation}\label{eqz}
F(w)=f^{-1}(w)=w+A_2w^2+A_3w^3+...
\end{equation}
or equivalently $f(F(w))=f(f^{-1}(w))=w$. Therefore, we have
\begin{equation}\label{eqz1}
w=(f^{-1}(w))+a_2(f^{-1}(w))^2+a_3(f^{-1}(w))^3+...
\end{equation}
In view of \eqref{eqz} and \eqref{eqz1}, the initial inverse coefficients become
\begin{align}
A_{2}=&-a_2, \label{eq11}\\
A_{3}=&-a_3+2a_2^2 \label{eq12},\\
A_{4}=&-a_4+5a_2a_3-5a_2^3,\label{eq13}\\
A_{5}=&-a_5+6a_2a_4-21a_2^2a_3+3a_3^2+14a_2^4.\label{eq14}
\end{align}
Ali \cite{Ali03} determined sharp estimates on inverse coefficients $A_2, A_3$ and Fekete-Szego functionals for the strongly starlike functions. For function $f\in\mathcal{S},$ 
the  logarithmic coefficients are defined  as
$log ({f(z)}/{z})=2\sum_{n=1}^{\infty}\gamma_nz^n,$ $z\in\mathbb{D}$.
Therefore, the initial logarithmic coefficients are given as
\begin{align}
\gamma_1=&\frac{1}{2}a_2 \label{eq2},\\
\gamma_2=&\frac{1}{2}(a_3-\frac{1}{2}a_2^2) \label{eq2a},\\
\gamma_3=&\frac{1}{2}(a_4-a_2a_3+\frac{1}{3}a_2^3)\label{eq2b}.
\end{align}
Further,   Ponnusamy et.al.~\cite{Ponnusamy}  discussed the logarithmic coefficients $\Gamma_n, n\in\mathbb{N}$ of inverse function $F$ defined by \eqref{eqz}, and given as 
\[log\bigg(\frac{F(w)}{w}\bigg)=2\sum_{n=1}^{\infty}\Gamma_nw^n,\quad  |w|<\frac{1}{4}.\] 
Here, $\Gamma_n$ are known as logarithmic inverse coefficients. The initial logarithmic  coefficient in term of inverse coefficients   are given as
\begin{align}
\Gamma_1=&\frac{1}{2}A_2 \label{eq40},\\
\Gamma_2=&\frac{1}{2}(A_3-\frac{1}{2}A_2^2) \label{eq41},\\
\Gamma_3=&\frac{1}{2}(A_4-A_2A_3+\frac{1}{3}A_2^3)\label{eq42}.
\end{align}
For $q,n\in\mathbb{N},$ the $q^{th}$ Hankel determinant involving initial coefficients of function $f\in \mathcal{A}$ is defined as 
\[H_{q,n}(f)=\begin{vmatrix}
a_{n} & a_{n+1} & \dots & a_{n+q-1}\\
a_{n+1}& a_{n+2} & \dots & a_{n+q}\\
\vdots & \vdots  & \dots & \vdots\\
a_{n+q-1}& a_{n+q} & \dots & a_{n+2(q-1)}
\end{vmatrix}\] 
where $a_1=1.$ In particular, for $q=2, n=2$ and $q=2, n=3$ we have
\begin{equation}\label{eq104}
H_{2,2}(f)=\begin{vmatrix}
a_{2} & a_{3} \\
a_{3}& a_{4} 
\end{vmatrix}=a_2a_4-a_3^2,
\quad
H_{2,3}(f)=\begin{vmatrix}
a_{3} & a_{4}\\
a_{4}& a_{5}
\end{vmatrix}=a_3a_5-a_4^2.
\end{equation}
Further, the $q^{th}$ Hankel determinant involving the coefficients of inverse function $F$ defined by \eqref{eqz}, is given as
\[H_{q,n}(f^{-1})=\begin{vmatrix}
A_{n} & A_{n+1} & \dots & A_{n+q-1}\\
A_{n+1}& A_{n+2} & \dots & A_{n+q}\\
\vdots & \vdots  & \dots & \vdots\\
A_{n+q-1}& A_{n+q} & \dots & A_{n+2(q-1)}
\end{vmatrix}.\] 
In particular, for $q=2, n=2$ and $q=2, n=3,$ we get
\begin{align}\label{eqegt}
H_{2,2}(f^{-1})=\begin{vmatrix}
A_{2} & A_{3} \\
A_{3}& A_{4} 
\end{vmatrix}=A_2A_4-A_3^2
\end{align}
\begin{align}\label{eqnin}
H_{2,3}(f^{-1})=\begin{vmatrix}
A_{3} & A_{4}\\
A_{4}& A_{5}
\end{vmatrix}=A_3A_5-A_4^2. 
\end{align}
In view of equation \eqref{eq11}, \eqref{eq12}, \eqref{eq13} and \eqref{eq14},  the expressions \eqref{eqegt} and \eqref{eqnin} become
\begin{align}
H_{2,2}(f^{-1})=&a_2a_4-a_3^2-a_2^2(a_3-a_2^2) \label{eq101}\\
H_{2,3}(f^{-1})=&a_3a_5-a_4^2-3a_3^3 \label{eq102}.
\end{align}
The $q^{th}$ Hankel determinant involving the logarithmic coefficient is given as
\[H_{q,n}(F_{f})=\begin{vmatrix}
\gamma_{n} & \gamma_{n+1} & \dots & \gamma_{n+q-1}\\
\gamma_{n+1}& \gamma_{n+2} & \dots & \gamma_{n+q}\\
\vdots & \vdots  & \dots & \vdots\\
\gamma_{n+q-1}& \gamma_{n+q} & \dots & \gamma_{n+2(q-1)}
\end{vmatrix}.\]
Specially,  for $q=2, n=1,$ we have
$H_{2,1}(F_{f}/2)=\begin{vmatrix}
\gamma_{1} & \gamma_{2} \\
\gamma_{2}& \gamma_{3} 
\end{vmatrix}=\gamma_1\gamma_3-\gamma_2^2.$ 
In terms of logarithmic inverse coefficient, we have  
\begin{align}
H_{2,1}(F_{f^{-1}}/2)=\begin{vmatrix}
\Gamma_{1} & \Gamma_{2} \\
\Gamma_{2}& \Gamma_{3} 
\end{vmatrix}=\Gamma_1\Gamma_3-\Gamma_2^2.
\end{align}
Further, in view of \eqref{eq11}, \eqref{eq12},  \eqref{eq13}, \eqref{eq40}, \eqref{eq41} and \eqref{eq42},  Hankel determinant $H_{2,1}(F_{f^{-1}}/2)$ becomes
\begin{align}\label{eqtwenty}
H_{2,1}(F_{f^{-1}}/2)=\frac{1}{48}\bigg(13a_2^4-12a_2^2a_3-12a_3^2+12a_2a_4\bigg).
\end{align}
For $q,n\in\mathbb{N},$  the Hermitian Toeplitz determinant for the function $f\in\mathcal{S}$ is defined as 
\[T_{q,n}(F_f)=\begin{vmatrix}
a_{n} & a_{n+1} & \dots & a_{n+q-1}\\
\bar{a}_{n+1}& a_{n} & \dots & a_{n+q-2}\\
\vdots & \vdots  & \dots & \vdots\\
\bar{a}_{n+q-1}& \bar{a}_{n+q-2} & \dots & a_{n}
\end{vmatrix},\] 
where $a_1=1$ and $\bar{a}_i=a_i, 1\leq i \leq n.$  In particular $T_{2,1}(F_{f})=
=1-|a_2|^2.$
Thus,  the second order Hermitian-Toeplitz determinant involving the logarithmic coefficient is written as
\[T_{2,1}(F_{f}/\gamma)=\begin{vmatrix}
\gamma_{1} & \gamma_{2} \\
\bar{\gamma}_{2}& \gamma_1 
\end{vmatrix}=\gamma_1^2-|\gamma_2|^2.\]
In view of equations \eqref{eq2} and \eqref{eq2a}, we have
\begin{equation}\label{eq105}
T_{2,1}(F_{f}/\gamma)=\frac{1}{16}(-a_2^4+4a_2^2+4a_2^2\Re a_3-4|a_3|^2).
\end{equation}
Further, the second order Hemitian Toeplitz determinanat involving the logarithmic inverse coefficients is written as 
\[T_{2,1}(F_{f^{-1}}/\Gamma)=\begin{vmatrix}
\Gamma_{1} & \Gamma_{2} \\
\bar{\Gamma}_{2}& \Gamma_1 
\end{vmatrix}=\Gamma_1^2-|\Gamma_2|^2.\]
In view of equations \eqref{eq40} and \eqref{eq41}, we have
\begin{equation}\label{eq107}
T_{2,1}(F_{f^{-1}}/\Gamma)=\frac{1}{16}(-A_2^4+4A_2^2+4A_2^2\Re A_3-4|A_3|^2).
\end{equation}
For more details, see~\cite{Hayman68,Girela2000,Roth2007,Kuc16,Kapoor2007}.

Initially, Pommerenke~\cite{Pomm66} investigated bounds on  Hankel determinants for the class $\mathcal{S}$. In~\cite{Lee13}, authors determined sharp bound on  second Hankel determinant of certain unified classes of univalent functions. Obradovic and Tuneshki~\cite{Tuneski23} determined sharp upper bounds for second order Hankel determinant for inverse functions of certain univalent functions. Authors~\cite{bog22} investigated  sharp bounds of the second Hankel determinant involving  logarithmic coefficients for the class of starlike and convex functions of order $\alpha$.
In \cite{Ponnusamy}, authors computed sharp bound on $\Gamma_n$ for the class $\mathcal{S}$ and its some well known subclasses. Further, the sharp bounds on second Hankel determinant $H_{2,1}(F_{f^{-1}}/2)$ involving logarithmic inverse coefficients for the starlike and convex functions were computed in \cite{Amal}.  Authors~\cite{vallu} determined sharp bound on
$H_{2,1}(F_{f}/2)$ for the classes of starlike and convex functions with respect to symmetric points.  Recently, Mandal et.al. \cite{mandal23} determined best possible bounds on second Hankel and Hemitian Toeplitz determinants involving logarithmic inverse
coefficients for the starlike and convex functions with respect to symmetric points.
Cudna et.al. \cite{Cudna} computed best possible estimates on $T_{2,1}(F_f)$ and $T_{3,1}(F_f)$ for starlike and convex functions of order $\alpha.$ $(\alpha\in [0,1)).$
In \cite{Tuneski} authors obtained sharp bounds on  $T_{2,1}(F_f)$ and $T_{3,1}(F_f)$ for class $\mathcal{S}$ and its subclasses.
For more details, see~\cite{Jast,sushilhj21,rai23}.

Motivated by the above discussed literature on Hankel as well as Hemitian Toeplitz determinants, in this paper, we investigate bounds on $|H_{2,1}(F_{f^{-1}}/2)|,$ $|H_{2,2}(f^{-1})|,$ $|H_{2,3}(f^{-1})|,$  $|H_{2,2}(f|,$ $|H_{2,2}(f^{-1})-H_{2,2}(f)|$ and $|H_{2,3}(f^{-1})-H_{2,3}(f)|$ for the functions $f$ belonging to the classes $\mathcal{S}_{S,e}^*$ and $\mathcal{S}_{S,L}^*$ respectively. Further, we determine bound on $T_{2,1}(	F_{f}/\gamma)$ and $T_{2,1}(	F_{f^{-1}}/\Gamma)$ for such classes. All obtained bounds are sharp. In addition, the invariant property is also discussed. 

\section{Hankel Determinants}
In 2022, Zaprawa~\cite{Zaprawa} computed sharp bounds on Zalcman functional $a_5-a_3^2$ and  Hankel determinants $H_{2,2}(f)$ and $H_{2,3}(f)$ for functions $f\in \mathcal{S}_{S,e}^*$.   
In this section, we determine the sharp estimate on second Hankel determinants with various coefficients  for the functions $f$ belonging to the classes $\mathcal{S}_{S,e}^*$ and $\mathcal{S}_{S,L}^*$ respectively. 
The following lemmas play an important role to demonstrate main results in this paper.
\begin{lem}\label{eqlem3}\cite{Sur1940}
Let $w(z)=c_1z+c_2z^2+c_3z^3+c_4z^4+...$ be a Schwarz function. Then 
\[|c_1|\leq1,\quad|c_2|\leq1-|c_1|^2,\quad|c_3|\leq1-|c_1|^2-\frac{|c_2|^2}{1+|c_1|},\quad|c_4|\leq1-|c_1|^2-|c_2|^2.\]
\end{lem}
\begin{lem}\label{lem1}\cite{Libera1983}
Let $\mathcal{P}$ be the class of analytic functions having the Taylor series of the form
\begin{equation}\label{pf}
p(z)=1+p_1z+p_2z^2+p_3z^3+\cdots
\end{equation}
satisfying the condition $\RE p(z)>0\;(z\in\mathbb{D})$. Then
\begin{align*}
2p_2=&p_1^2+t\xi,\\
4p_3=&p_1^3+2p_1t\xi-p_1t\xi^2+2t(1-|\xi|^2)\eta,\\
8p_4=&p_1^4+3p_1^2t\xi+(4-3p_1^2)t\xi^2+p_1^2t\xi^3+4t(1-|\xi|^2)(1-|\eta|^2)\gamma\\
&\qquad\qquad\,\,+4t(1-|\xi|^2)(p_1\eta-p_1\xi\eta-\bar{\xi}\eta^2),
\end{align*}
for some $\xi ,\eta,\gamma\in \overline{\mathbb{D}}$ and  $t=(4-p_1^2).$
\end{lem}

\begin{thm}\label{eqthm1}
If the function $f(z)=z+\sum_{n=2}^{\infty}a_nz^n\in \mathcal{S}_{S,e}^*,$   then 
\[|H_{2,1}(F_{f^{-1}}/2)|\leq\frac{1}{16}.\]
The inequality is sharp.
\end{thm}
\begin{proof}
If the function $f\in\mathcal{S}_{S,e}^*,$ then
\begin{equation}\label{eq5}
\frac{2zf'(z)}{f(z)-f(-z)}=e^{w(z)},
\end{equation}
where $w(z)=c_1z+c_2z^2+c_3z^3+c_4z^4+...$ is a Schwarz function defined on $\mathbb{D}$. On writing 
the Taylors series expansion of expression \eqref{eq5}, we get
\begin{align}\label{eq6}
1+2a_2z+2a_3z^2&+(-2a_2a_3+4a_4)z^3+(-2a_3^2+4a_5)z^4+...=1+c_1z+(\frac{c_1^2}{2}+c_2)z^2\notag\\&+(\frac{c_1^3}{6}+c_1c_2+c_3)z^3+\frac{1}{24}(c_1^4+12c_1^2c_2+12c_2^2+24c_1c_3+24c_4)z^4+\dots. 
\end{align}
On comparing the coefficients of identical  powers of $z$  in the expression \eqref{eq6},  we obtain 
\begin{align}
a_{2 }= &\frac{c_{1}}{2}, \label{eq7}\\
a_{3 }=&\frac{1}{4}(c_1^2+2c_2), \label{eq8}\\
a_{4 }= &\frac{1}{48}(5c_1^3+18c_1 c_2+12c_3), \label{eq9}\\
a_{5 }=& \frac{1}{24}(c_1^4+6c_1^2c_2 + 6c_2^2+6c_1c_3+6c_4). \label{eq10}
\end{align}
 On putting the values of $a_2, a_3$ and $a_4$ from  \eqref{eq7}, \eqref{eq8} and
 \eqref{eq9} in  \eqref{eqtwenty}, we have 
$$H_{2,1}(F_{f^{-1}}/2)=\frac{1}{768}(-c_1^4-48c_2^2-36c_1^2c_2+24c_1c_3).$$
On applying triangle inequality and Lemma \ref{eqlem3}, we get
\begin{align*}
|H_{2,1}(F_{f^{-1}}/2)|&\leq\frac{1}{768}(|c_1|^4+48|c_2|^2+36|c_1|^2|c_2|+24|c_1|\bigg(1-|c_1|^2-\frac{|c_2|^2}{1+|c_1|}\bigg)\notag\\&=\chi(|c_1|,|c_2|).
\end{align*}
Next, on setting $|c_1|=x, |c_2|=y,$ we get
$$\chi(x,y)=\frac{1}{768}\bigg(x^4+48y^2+36x^2y+24x\left(1-x^2-\frac{y^2}{1+x}\right)\bigg).$$
 Now, we will find  the maximum value of the function $\chi(x,y)$ over the region $\Lambda=\{(x,y):0\leq x \leq 1, 0\leq y \leq 1-x^2\}.$ Since ${\partial \chi}/{\partial y}=\bigg(48y(2-\frac{x}{1+x})+36x^2\bigg)/768,$ which has no solution in the interior of $\Lambda.$ The function $\chi$ has no maximum value in the interior of $\Lambda.$
Next, we consider the boundary of $\Lambda$ such that 
\begin{align*}
\chi(0,y)&=\frac{48}{768}y^2\leq\frac{1}{16},\quad 0\leq y\leq1,\\
\chi(x,0)&=\frac{1}{768}(x^4-24x^3+24x)\leq0.0122, \quad 0\leq x\leq1,\\
\chi(x,1-x^2)&=\frac{1}{768}(-11x^4-36x^2+48)\leq\frac{1}{16},\quad 0\leq x\leq1.
\end{align*}
Therefore, we obtain the  best possible estimate 1/16 on $|H_{2,1}(F_{f^{-1}}/2)|.$ The function $f_1$ defined by \eqref{eqt1}
works as an extremal function for the inequality $|H_{2,1}(F_{f^{-1}}/2)|\leq\frac{1}{16}.$\qedhere
\end{proof}
\begin{rem}
From ~\cite{kumar23} and Theorem \ref{eqthm1}, it is noted that the bounds on  $|\gamma_1\gamma_3-\gamma_2^2|$ and $|\Gamma_1\Gamma_3-\Gamma_2^2|$ are invariant for the function $f\in\mathcal{S}_{S,e}^*.$\qedhere
\end{rem}

\begin{thm}\label{eqthrm3}
If the function $f(z)=z+\sum_{n=2}^{\infty}a_nz^n\in \mathcal{S}_{S,e}^*,$ then 
\[|H_{2,2}(f^{-1})|\leq\frac{1}{4}.\]
The inequality is sharp.
\end{thm}
\begin{proof}
Since $f\in \mathcal{S}_{S,e}^*,$ on substituting the value of $a_2, a_3$ and $a_4$ from  \eqref{eq7}, \eqref{eq8} and \eqref{eq9} respectively in expression \eqref{eq101}, we get 
\[H_{2,2}(f^{-1})=\frac{1}{96}(-c_1^4-18c_1^2 c_2-24c_2^2+12c_1 c_3).\]
On applying Lemma \eqref{eqlem3}, the above expression becomes
$$|H_{2,2}(f^{-1})|\leq\frac{1}{96}\bigg(|c_1|^4+18|c_1|^2|c_2|+24|c_2|^2+12|c_1|\left(1-|c_1|^2-\frac{|c_2|^2}{1+|c_1|}\right)\bigg).$$
On letting $x=|c_1|$ and $y=|c_2|,$ we have
 $|H_{2,2}(f^{-1})|\leq\Upsilon(x,y),$
where \[\Upsilon(x,y)=\frac{1}{96}\bigg(x^4+18x^2y+24y^2+12x\left(1-x^2-\frac{y^2}{1+x}\right)\bigg).\] We determine the maximum value of the function $\Upsilon$ over the region $\Lambda=\{(x,y):0\leq x \leq 1, 0\leq y \leq 1-x^2\}.$
We consider the two cases on $\Lambda,$ which are discussed as
\begin{itemize}
\item [$(A_1)$] The solution of the system of equations ${\partial \Upsilon}/{\partial x}=0$ and ${\partial \Upsilon}/{\partial y}=0$ gives a critical point of the function $\Upsilon.$ On solving  ${\partial \Upsilon}/{\partial y}=18x^2+48y-({24xy}/{(1+x)})=0,$ we get
\[18x^2+24y\left(2-\frac{x}{1+x}\right)=0\] which has no solution for $0<x<1$ and $0<y<1-x^2.$ Therefore, the function has no maximum value.
\item [$(A_2)$] The continuity of the function $\Upsilon$ on the compact region $\Lambda$ ensure that the  maximum value of $\Upsilon$ is attained at boundary of $\Lambda.$
\begin{enumerate}
\item [(i)]$\Upsilon(0,y)={24y^2}/{96}\leq{1}/{4}$ for all $0\leq y\leq1.$
\item[(ii)]$\Upsilon(x,0)=(x^4-12x^3+12x)/96\leq0.049$ for all $0\leq x\leq1.$
\item[(iii)]$\Upsilon(x,1-x^2)=(-5x^4-18x^2+24)/96\leq{1}/{4}$ for all $0\leq x\leq1.$
\end{enumerate}
\end{itemize}
From the cases $(A_1)$ and $(A_2),$ we get the desired  estimate  on $|H_{2,2}(f^{-1})|.$ Sharpness follows for the function $f_1$ defined by \eqref{eqt1}.\qedhere
\end{proof}
\begin{rem}
In view of \cite{Zaprawa} Theorem-4, pp.-17 and  Theorem \ref{eqthrm3}, the bound on $|H_{2,2}(f)|$ and $|H_{2,2}(f^{-1})|$ are invariant for the function $f\in\mathcal{S}_{S,e}^*.$\qedhere
\end{rem}
\begin{cor}
If $f(z)=z+\sum_{n=2}^{\infty}a_nz^n\in \mathcal{S}_{S,e}^*,$ then 
$|H_{2,2}(f^{-1})-H_{2,2}(f)|\leq\frac{1}{32}.$
\end{cor}
 \begin{proof}
Since $f\in \mathcal{S}_{S,e}^*$ and  $H_{2,2}(f^{-1})=A_2A_4-A_3^2=H_{2,2}(f)-a_2^2(a_3-a_2^2),$  we have
\[|H_{2,2}(f^{-1})-H_{2,2}(f)|=|a_2^2(a_3-a_2^2)|\leq|a_2^2| |(a_3-a_2^2)|.\]
Using the equations \eqref{eq7} and \eqref{eq8} we have
\[|H_{2,2}(f^{-1})-H_{2,2}(f)|\leq\frac{1}{8}|c_1|^2(1-|c_1|^2)\leq\frac{1}{32},\quad 0\leq|c_1|\leq1.\qedhere\]
 \end{proof}

\begin{thm}
If the function $f(z)=z+\sum_{n=2}^{\infty}a_nz^n\in \mathcal{S}_{S,e}^*,$ then 
$|H_{2,3}(f^{-1})|\leq\frac{1}{4}.$
The inequality is sharp.
\end{thm}
\begin{proof}
Since $f \in \mathcal{S}_{S,e}^*,$ in view of \eqref{eq8}, \eqref{eq9}, \eqref{eq10} and \eqref{eq102},  we have
\begin{align*}
|H_{2,3}(f^{-1})|\leq\frac{109}{2304}{|c_1|^6}+\frac{53}{192}{|c_1|^4|c_2|}+\frac{33}{64}{|c_1|^2|c_2|^2}&+\frac{1}{4}{|c_2|^3}+\frac{1}{96}{|c_1|^3|c_3|}+\frac{1}{16}{|c_1| |c_2| |c_3|}\\&+\frac{1}{16}{|c_3|^2}+\frac{1}{16}{|c_1|^2|c_4|}+\frac{1}{8}{|c_2| |c_4|}.
\end{align*}
By making use of Lemma \eqref{eqlem3} in the above expression, we get
\begin{align}\label{eq16}
|H_{2,3}(f^{-1})|\leq\frac{109}{2304}{|c_1|^6}&+\frac{53}{192}{|c_1|^4|c_2|}+\frac{33}{64}{|c_1|^2|c_2|^2}+\frac{1}{4}{|c_2|^3}+\frac{1}{96}{|c_1|^3\bigg(1-|c_1|^2-\frac{|c_2|^2}{1+|c_1|}\bigg)}\notag\\&+\frac{1}{16}{|c_1| |c_2|\bigg(1-|c_1|^2-\frac{|c_2|^2}{1+|c_1|}\bigg)}+
\frac{1}{16}{\bigg(1-|c_1|^2-\frac{|c_2|^2}{1+|c_1|}\bigg)^2}\notag\\&+\frac{1}{16}{|c_1|^2(1-|c_1|^2-|c_2|^2)}+\frac{1}{8}{|c_2|(1-|c_1|^2-|c_2|^2)}.
\end{align}
Next, setting $x=|c_1|$ and $y=|c_2|,$ the above inequality \eqref{eq16} becomes $|H_{2,3}(f^{-1})|\leq\Phi(x,y)$ 
where 
\begin{align*}
\Phi(x,y)=\frac{109}{2304}{x^6}&+\frac{53}{192}{x^4y}+\frac{33}{64}{x^2y^2}+\frac{1}{4}{y^3}+\frac{1}{96}{x^3\bigg(1-x^2-\frac{y^2}{1+x}\bigg)}\\&+\frac{1}{16}{x y\bigg(1-x^2-\frac{y^2}{1+x}\bigg)}+
\frac{1}{16}{\bigg(1-x^2-\frac{y^2}{1+x}\bigg)^2}\\&+\frac{1}{16}{x^2(1-x^2-y^2)}+\frac{1}{8}{y(1-x^2-y^2)}.
\end{align*}
Further, we find the maximum value of the function $\Phi$ in the region $\Lambda=\{(x,y):0\leq x \leq 1, 0\leq y \leq 1-x^2\}.$ We  consider following two cases.
\begin{itemize}
\item [$(A_1)$] In this case, we compute the maximum value of $\Phi$ on the boundary of $\Lambda$. A basic calculation on $\Phi$ gives 
\begin{align*}
\Phi(0,y)&=\frac{1}{16}{(y^4+2y^3-2y^2+2y+1)}\leq\frac{1}{4},\quad 0\leq y\leq1.\\
\Phi(x,0)&=\frac{1}{2304}{(109x^6-24x^5+24x^3-144x^2+144)}\leq\frac{1}{16}, \quad 0\leq x\leq1.\\
\Phi(x,1-x^2)&=\frac{1}{2304}(493x^6-996x^4+36x^2+576)\leq\frac{1}{4},\quad 0\leq x\leq1.
\end{align*}
\item [$(A_2)$] We consider the interior point of $\Lambda.$ Since critical point is a solution of system of equations ${\partial \Phi}/{\partial x}=0$ and ${\partial \Phi}/{\partial y}=0,$
thus
\begin{align}\label{eq16}
\frac{\partial \Phi}{\partial x}=\frac{109}{384}{x^5}&+\frac{x^2y}{48}{(53x-9)}+\frac{29}{32}{xy^2}+\frac{1}{96}{(-5x^4-24x^3+3x^2+12x)}\notag\\&+\frac{1}{16}{y(1-4x)}-\frac{1}{96(1+x)^2}{(3x^2y^2+2x^3y^2+6y^3)}\notag\\&+\frac{1}{8}\bigg(1-x^2-\frac{y^2}{1+x}\bigg)\bigg(-2x+\frac{y^2}{(1+x)^2}\bigg)=0
\end{align}
and
\begin{align}\label{eq17}
\frac{\partial \Phi}{\partial y}=\frac{53}{192}{x^4}&+\frac{29}{32}{x^2y}+\frac{3}{4}{y^2}-\frac{x^3y}{48(1+x)}+\frac{1}{16}\bigg(x-x^3-\frac{3xy^2}{1+x}\bigg)\notag\\&+\frac{1}{4(1+x)^2}(y^3+x^3y+x^2y-xy-y)\notag\\&+\frac{1}{8}(1-x^2-3y^2)=0.
\end{align}
Therefore, the system of equation \eqref{eq16} and \eqref{eq17} has no common  solution in  the interior of $\Lambda$ and we note that the function $\Phi$ has  no maximum value.
\end{itemize} 
From the cases ($A_1$) and ($A_2$),  we conclude that the maximum value of $\Phi$ on region $\Lambda$ is $1/4.$  Sharpness follows for the function $f_1$ defined in \eqref{eqt1}.\qedhere
\end{proof}
\begin{cor}
If the function $f(z)=z+\sum_{n=2}^{\infty}a_nz^n\in \mathcal{S}_{S,e}^*,$ then we have
\[|H_{2,3}(f^{-1})-H_{2,3}(f)|\leq\frac{3}{8}.\]
The inequality is sharp  for the function $f_1$ defined by \eqref{eqt1}.
\end{cor}
\begin{proof}
Since $f \in \mathcal{S}_{S,e}^*$ and  $H_{2,3}(f^{-1})=a_3a_5-a_4^2-3a_3^3=H_{2,3}(f)-a_3^3.$
Therefore, 
$|H_{2,3}(f^{-1})-H_{2,3}(f)|=3|a_3^3|.$
Using \eqref{eq8}, we get
\begin{align*}
|H_{2,3}(f^{-1})-H_{2,3}(f)|&=\frac{3}{64}\bigg|c_1^6+8c_2^3+6c_1^4c_2+12c_1^2c_2^2\bigg|\notag\\&\leq\frac{3}{64}\bigg(|c_1|^6+8|c_2|^3+6|c_1|^4|c_2|+12|c_1|^2|c_2|^2\bigg)\notag\\&=u (|c_1|,|c_2|).
\end{align*} 
On similar lines as done in the previous theorem, it is easy to obtain the maximum value of the function $u(|c_1|, |c_2|)$ on the region  $\Lambda=\{(|c_1|,|c_2|):0\leq |c_1| \leq 1, 0\leq |c_2| \leq 1-|c_1|^2\}.$ 
\qedhere
\end{proof}

\begin{thm}\label{eqthrm8}
If the function $f(z)=z+\sum_{n=2}^{\infty}a_nz^n\in \mathcal{S}_{S,	L}^*,$ then
 \[|H_{2,1}(F_{f^{-1}}/2)|\leq\frac{1}{64}.\]
The inequality is sharp.
\end{thm}
\begin{proof}
Let  the function $f\in\mathcal{S}_{S,L}^*$. Then  
\begin{equation}\label{eqrs}
\frac{2zf'(z)}{f(z)-f(-z)}=\sqrt{1+w(z)},\quad\mbox{for all}\,\, z\in\mathbb{D}
\end{equation}
where
$w(z)$ is the Schwarz function in $\mathbb{D}.$
Since 
$p(z)={(1+w(z))}/{(1-w(z))},$ then \eqref{eqrs} becomes
\begin{equation}\notag
\frac{2zf'(z)}{f(z)-f(-z)}=
1+\frac{1}{4}{p_1} z+ \frac{1}{32} (8 {p_2}-5 {p_1}^2)z^2+\frac{1}{128} (13  {p_1}^3-40  {p_1} {p_2}+32 {p_3})z^3+\cdots.
\end{equation}
On equating the coefficients of like powers of $z$ on both the sides,  we have
\begin{align}
a_2=&\frac{1}{8}p_1 \label{eq52},\\
a_3=&\frac{1}{64}(-5p_1^2+8p_2)\label{eq53},\\
a_4=&\frac{1}{1024}(21p_1^3-72p_1p_2+64p_3)\label{eq54}\\
a_5=&\frac{1}{8192}(-116p_1^4+544p_1^2p_2-256p_2^2-640p_1p_3+512p_4).\label{eq55}
\end{align}
 On substituting the values of $a_2, a_3$ and $a_4$ from
\eqref{eq52}, \eqref{eq53} and  \eqref{eq54} respectively in \eqref{eqtwenty}, we get
$$H_{2,1}(F_{f^{-1}}/2)=\frac{1}{393216}(-202p_1^4+864p_1^2p_2-1536p_2^2+768p_1p_3).$$
We assume  $p_1=p$ and by making use of Lemma \eqref{lem1},  we have
\[H_{2,1}(F_{f^{-1}}/2)=\frac{1}{393216}\bigg(38p^4+48p^2(4-p^2)(\xi-4\xi^2)+384p(4-p^2)(1-|\xi|^2)\eta-384(4-p^2)^2\xi^2\bigg)\]
such that
\begin{align*}
|H_{2,1}(F_{f^{-1}}/2)|\leq&\frac{1}{393216}\bigg(38p^4+48p^2(4-p^2)(|\xi|+4|\xi|^2)\\
&\quad\quad \qquad\,\, +384p(4-p^2)(1-|\xi|^2)|\eta|
+384(4-p^2)^2|\xi|^2\bigg)\notag\\
=&N(p,|\xi|,|\eta|).
\end{align*}
On  setting $x=|\xi|, y=|\eta| $, we have
\[N(p,x,y)=\frac{1}{393216}\bigg(38p^4+48p^2(4-p^2)(x+4x^2)+384py(4-p^2)(1-x^2)+384(4-p^2)^2x^2\bigg)\]
 for $x, y\in[0,1]$ and $p\in[0,2].$
Next, we find the maximum value of  function $N(p,x,y)$ over the cuboid   $\Omega=[0,2]\times[0,1]\times[0,1].$ We consider the following three cases on the cuboid $\Omega.$
\begin{itemize}
\item [(A)]  Let $(p,x,y)\in\Omega.$ Since ${\partial N}/{\partial y}=(p(4-p^2)(1-x^2))/1024\neq0,$ the function $N$ has no critical point in the interior of $\Omega.$ 
\item [(B)] We   consider  the interior of six faces of the cuboid $\Omega.$
\begin{enumerate}
\item [(i)] At $p=0,$ we have
$$N(0,x,y)=\frac{6144x^2}{393216}\leq\frac{1}{64},\quad x,y\in(0,1).$$
\item [(ii)] At $p=2,$ we have
$$N(2,x,y)=\frac{19}{12288},\quad x,y\in(0,1).$$
\item [(iii)] At $x=0,$ we have 
$N(p,0,y)=(384py(4-p^2))/393216.$ Since ${\partial N(p,0,y)}/{\partial y}=(p(4-p^2)(1-x^2))/1024\neq0.$ The function $N(p,0,y)$ has no critical point.
\item [(iv)] At $x=1,$ we have 
$$N(p,1,y)=\frac{1}{393216}(528p^4-2112p^2+6144)\leq\frac{1}{64},\quad c\in(0,2).$$
\item [(v)] At $y=0,$ we have 
$N(p,x,0)=\bigg(38p^4+48p^2(4-p^2)(x+4x^2)+384(4-p^2)^2x^2\bigg)/393216.$
Since ${\partial N(p,x,0)}/{\partial x}=\bigg(48p^2(4-p^2)(1+8x)+768(4-p^2)^2x\bigg)/393216\neq0,$ the function $N(p,x,0)$ has no maximum value for $(p,x)\in(0,2)\times(0,1).$
\item [(vi)] At $y=1,$ we have $N(p,x,1)=\bigg(38p^4+48p^2(4-p^2)(x+4x^2)+384p(4-p^2)(1-x^2)+384(4-p^2)^2x^2\bigg)/{393216}.$ Now,
$$\frac{\partial N(p,x,1)}{\partial x}=\frac{1}{8192}\bigg((4-p^2)(p^2(1-8x)-16x(p-4))\bigg)=0,$$
has no solution for $(p,x)\in(0,2)\times(0,1).$ Therefore, the system of equation ${\partial N(p,x,1)}/{\partial p}=0$ and ${\partial N(p,x,1)}/{\partial x}=0$ has no  solution in $(0,2)\times(0,1).$ Hence, the  function $N(p,x,1)$ has no critical point.
\end{enumerate}
\item [(C)]Next, we consider the following twelve edges of the cuboid $\Omega$
\begin{enumerate}
\item [(i)]$N(p,0,0)=N(2,1,y)=N(2,0,y)=N(2,x,0)=N(2,x,1)={19}/{12288},$
\item [(ii)] $N(p,0,1)=(38p^4-384p^3+1536p)/393216\leq0.0032.$
\item [(iii)]
$N(p,1,1)=N(p,1,0)\leq{1}/{64},$
\item [(iv)]
$N(0,0,y)=0,$
\item [(v)]
$N(0,1,y)=1/64,$
\item [(vi)]
$N(0,x,0)=N(0,x,1)=x^2/64\leq1/64.$
\end{enumerate}
\end{itemize}
In view of cases (A), (B) and (C), we obtain the maximum value 1/64 of the function $N(p,x,y)$ over $\Omega$ which  is sharp for the function $g_1$ defined by \eqref{eqt7}.\qedhere
\end{proof}
\begin{rem}
From ~\cite{kumar23} and Theorem \ref{eqthrm8}, it is noted that the bounds on  $|\gamma_1\gamma_3-\gamma_2^2|$ and $|\Gamma_1\Gamma_3-\Gamma_2^2|$ are invariant for the function  $f\in\mathcal{S}_{S,L}^*.$
\end{rem}
Next result shows that the bound on $|H_{2,2}(f)|$ and $|H_{2,2}(f^{-1})|$ are invariant for the functions $f\in\mathcal{S}_{S,L}^*.$
\begin{thm}
If the function  $f(z)=z+\sum_{n=2}^{\infty}a_nz^n\in \mathcal{S}_{S,L}^*,$ then 
\[|H_{2,2}(f)|, |H_{2,2}(f^{-1})|\leq\frac{1}{16}.\]
The inequality is sharp.
\end{thm}
\begin{proof}
Let the function $f\in\mathcal{S}_{S,L}^*.$ Then,  we have
\begin{equation}\label{eq76}
\frac{2zf'(z)}{f(z)-f(-z)}=\sqrt{1+(w(z))},\quad z\in\mathbb{D},
\end{equation}
where $w(z)=c_1z+c_2z^2+c_3z^3+c_4z^4+...$ is a Schwarz function defined on $\mathbb{D}.$
Using the Taylors expansion in equation \eqref{eq76}, we have
\begin{align}\label{eq77}
\frac{2zf'(z)}{f(z)-f(-z)}=1+&\frac{c_1}{2}z+(-\frac{c_1^2}{8}+\frac{c_2}{2})z^2+\frac{1}{16}(c_1^3-4c_1c_2+8c_3)z^3\notag\\& +\frac{1}{128}(-5c_1^4+24c_1^2c_2-16c_2^2-32c_1c_3+64c_4)z^4+\cdots. 
\end{align}
On comparing the coefficients of identical  powers of $z$  in the expression \eqref{eq77},  we obtain 
\begin{align}
a_{2 }= &\frac{c_{1}}{4}, \label{eq78}\\
a_{3 }=&\frac{1}{16}(-c_1^2+4c_2), \label{eq79}\\
a_{4 }= &\frac{1}{128}(c_1^3-4c_1 c_2+16c_3), \label{eq80}\\
a_{5 }=& \frac{1}{128}(-c_1^4+4c_1^2c_2 - 8c_1c_3+16c_4). \label{eq81}
\end{align}
By using  \eqref{eq78}, \eqref{eq79} and \eqref{eq80} in \eqref{eq104}, we have
\begin{align*}
|H_{2,2}(f)|&=\frac{1}{512}\bigg|(-c_1^4+12c_1^2c_2-32c_2^2+16c_1c_3)\bigg|\notag\\&\leq\frac{1}{512}(|c_1|^4+12|c_1|^2|c_2|+32|c_2|^2+16|c_1| |c_3|).
\end{align*}
On using Lemma \eqref{eqlem3} in the above expression, we get
\begin{align*}
|H_{2,2}(f)|&\leq\frac{1}{512}\bigg(|c_1|^4+12|c_1|^2|c_2|+32|c_2|^2+16|c_1|\left(1-|c_1|^2-\frac{|c_2|^2}{1+|c_1|}\right)\bigg)\notag\\&=\beta(|c_1|,|c_2|). 
\end{align*}
We set $x=|c_1|, y=|c_2|.$ Then, we have 
\[\beta(|c_1|,|c_2|)=\frac{1}{512}\bigg(x^4+12x^2y+32y^2+16x\left(1-x^2-\frac{y^2}{1+x}\right)\bigg).\] 
Next, we find  the maximum value of function  $\beta(x,y)$ on  the region 
$\Lambda=\{(x,y):0\leq x \leq 1,\,\, 0\leq y \leq 1-x^2\}.$ Since ${\partial \beta}/{\partial y}=12x^2+32y\left(2-\frac{x}{1+x}\right)\neq0$ in the interior of $\Lambda.$ The function $\beta$ has no critical point in the interior of $\Lambda.$ Next, we  consider the boundary of $\Lambda$. A simple calculation gives
\begin{enumerate}
\item [(i)]$\beta(0,y)={32y^2}/{512}\leq{1}/{16}$ for all $0\leq y\leq1,$
\item[(ii)]$\beta(x,0)=(x^4-16x^3+16x)/512\leq0.0123$ for all $0\leq x\leq1,$
\item[(iii)]$\beta(x,1-x^2)=(5x^4-36x^2+32)/512\leq{1}/{16}$ for all $0\leq x\leq1.$
\end{enumerate}
Therefore, the  estimate on $|H_{2,2}(f)|$  is 1/16, which is sharp  for the function $g_1$ defined in \eqref{eqt7}.
Next, on substituting  values of $a_2$, $a_3$ and $a_4$ from  \eqref{eq78}, \eqref{eq79} and \eqref{eq80} respectively in  \eqref{eq101},  we have 
\[512H_{2,2}(f^{-1})=3c_1^4+4c_1^2c_2+16c_1c_3-32c_2^2.\] 
In view of  Lemma \eqref{eqlem3},  above expression becomes
\begin{align*}
|H_{2,2}(f^{-1})|&\leq\frac{1}{512}\bigg(3|c_1|^4+4|c_1|^2|c_2|+32|c_2|^2+16|c_1|(1-|c_1|^2-\frac{|c_2|^2}{1+|c_1|})\bigg)\notag\\&=\alpha(|c_1|,|c_2|).
\end{align*}
We set $x=|c_1|$ and $y=|c_2|$. Then $|H_{2,2}(f^{-1})|\leq\alpha(x,y)$
where $\alpha(x,y)=\bigg(3x^4+4x^2y+32y^2+16x\left(1-x^2-\frac{y^2}{1+x}\right)\bigg)\bigg/512.$
 Next,  we  find the maximum value of the function $\alpha(x,y)$ over $\Lambda.$ We consider the following two cases.
\begin{itemize}
\item [$(A_1)$] We are taking the interior point of $\Lambda.$ The solution of the system of equation ${\partial \alpha(x,y)}/{\partial x}=0$ and ${\partial \alpha(x,y)}/{\partial y}=0$ refers as a critical point of the function $\alpha.$ On solving  ${\partial \alpha}/{\partial y}=4x^2+64y-(32xy/(1+x))=0,$ we have
\[4x^2+32y(2-\frac{x}{1+x})=0,\] which is not valid   in the interior of $\Lambda.$ Thus, the function $\alpha$ has no maximum value in the interior of $\Lambda.$
\item [$(A_2)$] Next, we consider the boundary point of  $\Lambda$ 
\begin{enumerate}
\item [(i)]$\alpha(0,y)={32y^2}/{512}\leq{1}/{16}$ for all $0\leq y\leq1,$
\item[(ii)]$\alpha(x,0)=(3x^4-16x^3+16x)/512\leq0.0128$ for all $0\leq x\leq1,$
\item[(iii)]$\alpha(x,1-x^2)=(15x^4-44x^2+32)/512\leq{1}/{16}$ for all $0\leq x\leq1.$
\end{enumerate}
\end{itemize}
Therefore, we obtain the bound 1/16 on $|H_{2,2}(f^{-1})|,$   which is sharp and sharpness follows by the function $g_1$ defined in \eqref{eqt7}.\qedhere
\end{proof}

\begin{cor}
If the function  $f(z)=z+\sum_{n=2}^{\infty}a_nz^n\in \mathcal{S}_{S,L}^*,$ then 
\[|H_{2,2}(f^{-1})-H_{2,2}(f)|\leq\frac{1}{128}.\]
The inequality is sharp.
\end{cor}
\begin{proof}
Since $f\in \mathcal{S}_{S,L}^*$ and $H_{2,2}(f^{-1})=A_2A_4-A_3^2=H_{2,2}(f)-a_2^2(a_3-a_2^2),$ then 
by using \eqref{eq78} and \eqref{eq79}, we get
\[|H_{2,2}(f^{-1})-H_{2,2}(f)|\leq|a_2^2| |(a_3-a_2^2)|\leq\frac{1}{128}(|c_1|^4+2|c_1|^2|c_2|)\leq\frac{1}{128},\]
where $0\leq|c_1|\leq1$ and $0\leq |c_2|\leq 1-|c_1|^2.$ The inequality is sharp for the  function $g_2$ given by \eqref{eqt72}.\qedhere
\end{proof}

\begin{thm}
If the function $f(z)=z+\sum_{n=2}^{\infty}a_nz^n\in \mathcal{S}_{S,L}^*,$ then $|H_{2,3}(f^{-1})|\leq\dfrac{3}{64}.$
The inequality is sharp.
\end{thm}
\begin{proof}
Since $f\in \mathcal{S}_{S,L}^*,$ from  expressions \eqref{eq79}, \eqref{eq80} and \eqref{eq81}, we have
\begin{align}
a_3a_5=&\frac{1}{2048}(c_1^6-8c_1^4c_2+16c_1^2c_2^2+8c_1^3c_3-32c_1c_2c_3-16c_1^2c_4+64c_2c_4),\label{eq90}\\
a_4^2=&\frac{1}{16384}(c_1^6+16c_1^2c_2^2-8c_1^4c_2-128c_1c_2c_3+256c_3^2+32c_1^3c_3),\label{eq91}\\
3a_3^3=&\frac{3}{4096}(-c_1^6+64c_2^3+12c_1^4c_2-48c_1^2c_2^2).\label{eq92}
\end{align}
In view of  \eqref{eq90}, \eqref{eq91}, \eqref{eq92} and \eqref{eq102}, we get
\begin{align*}
16384H_{2,3}(f^{-1})=  (19c_1^6-200c_1^4c_2+688c_1^2c_2^2-768c_2^3+32c_1^3c_3&-128c_1c_2c_3-256c_3^2\notag\\&-128c_1^2c_4+512c_2c_4).
\end{align*}
On applying Lemma \ref{eqlem3} in  the above expression, 
 we get
\begin{align}\label{eq95}
|H_{2,3}(f^{-1})|\leq \frac{1}{16384} \bigg(19|c_1|^6&+200|c_1|^4|c_2|+688|c_1|^2|c_2|^2+768|c_2|^3\notag\\&+32|c_1|^3(1-|c_1|^2-\frac{|c_2|^2}{1+|c_1|})\notag\\&+128|c_1| |c_2| (1-|c_1|^2-\frac{|c_2|^2}{1+|c_1|})\notag\\&+256(1-|c_1|^2-\frac{|c_2|^2}{1+|c_1|})^2\notag\\&+128|c_1|^2(1-|c_1|^2-|c_2|^2)\notag\\&+512|c_2|(1-|c_1|^2-|c_2|^2)\bigg).
\end{align}
On  letting $x=|c_1|$ and $y=|c_2|,$ the equation \eqref{eq95} becomes
$|H_{2,3}(f^{-1})|\leq\mu(x,y)$
where 
\begin{align*}
\mu(x,y)=\frac{1}{16384} \bigg(19x^6&+200x^4y+688x^2y^2+768y^3+32x^3(1-x^2-\frac{y^2}{1+x})\notag\\&+128xy (1-x^2-\frac{y^2}{1+x})+256(1-x^2-\frac{y^2}{1+x})^2\notag\\&+128x^2(1-x^2-y^2)+512y(1-x^2-y^2)\bigg).
\end{align*}
We find the maximum value of the function $\mu(x,y)$ over the region $\Lambda=\{(x,y):0\leq x \leq 1, 0\leq y \leq 1-x^2\}.$  We  consider following two cases:
\begin{itemize}
\item [$(A_1)$] On boundary point of $\Lambda,$ we have
\begin{align*}
\mu(0,y)=&\frac{1}{16384}{(256y^4+256y^3-512y^2+512y+256)}\leq\frac{3}{64},\quad 0\leq y\leq1,\\
\mu(x,0)=&\frac{1}{16384}\bigg(19x^6+(1-x^2)(32x^3-128x^2+256)\bigg)\leq\frac{1}{64}, \quad 0\leq x\leq1,\\
\mu(x,1-x^2)=&\frac{1}{16384}(19x^6+768(1-x^2)^3+1584x^2(1-x^2)^2\\
&\qquad\qquad\qquad\qquad\quad\qquad\quad+360x^4(1-x^2))\leq\frac{3}{64},\quad 0\leq x\leq1.
\end{align*}
\item [$(A_2)$] In the interior of  $\Lambda$, the equation   ${\partial \mu}/{\partial y}=0$  gives 
\begin{align*}
200x^4-\frac{64x^3y}{1+x}+1120x^2y+128x(1-x^2-\frac{3y^2}{1+x})&-\frac{1024y}{1+x}(1-x^2-\frac{y^2}{1+x})\notag\\&+512(1-x^2)+768y^2=0
\end{align*}
which has no solution in the interior of $\Lambda$. 
Therefore the function $\mu$ has  no maximum value.
\end{itemize} 
Hence,  we conclude that the maximum value of the function $\mu$ is $3/64.$ Sharpness follows for the function $g_1$ defined in \eqref{eqt7}.\qedhere
\end{proof}

\begin{cor}
If the function $f(z)=z+\sum_{n=2}^{\infty}a_nz^n\in \mathcal{S}_{S,L}^*,$ then $$|H_{2,3}(f^{-1})-H_{2,3}(f)|\leq\frac{3}{64}.$$
The inequality is sharp for the function $g_1$ defined by \eqref{eqt7}.
\end{cor}
\begin{proof}
Since $f\in \mathcal{S}_{S,L}^*,$ and $H_{2,3}(f^{-1})=A_3A_5-A_4^2=H_{2,3}(f)-3a_3^3$,  then
using \eqref{eq79},  we get
\begin{align*}
|H_{2,3}(f^{-1})-H_{2,3}(f)|&=\frac{3}{4096}\bigg|-c_1^6+64c_2^3+12c_1^4c_2-48c_1^2c_2^2\bigg|\notag\\
&\leq\frac{3}{4096}\bigg(|c_1|^6+64|c_2|^3+12|c_1|^4|c_2|+48|c_1|^2|c_2|^2\bigg)=v(|c_1|,|c_2|).
\end{align*} 
Proceeding as similar lines of previous theorem, we get the maximum value of the function $v(|c_1|, |c_2|)$ in the region  $\Lambda=\{(|c_1|,|c_2|):0\leq |c_1| \leq 1, 0\leq |c_2| \leq 1-|c_1|^2\}.$  \qedhere
\end{proof}

\section{Hermitian-Toeplitz  determinants }
Recently, authors~\cite{ckms22} computed sharp bounds on third order Hermitian-Toeplitz determinants involving initial coefficients for the classes $\mathcal{S}^*_{S, \,e}$ and $\mathcal{S}^*_{S,\, L}$. In this section, we determine sharp bound on second Hermitian-Toeplitz determinants involving initial logarithmic coefficients and logarithmic inverse coefficients for the classes $\mathcal{S}^*_{S, \,e}$ and $\mathcal{S}^*_{S,\, L}$. Next result show that the bound on $T_{2,1}(	F_{f}/\gamma)$ and $T_{2,1}(	F_{f^{-1}}/\Gamma)$ are invariant for the functions $f\in\mathcal{S}_{S,e}^*.$
\begin{thm}
If the function $f\in \mathcal{S}_{S,e}^*,$ then 
\[-\frac{1}{16}\leq T_{2,1}(	F_{f}/\gamma), T_{2,1}(	F_{f^{-1}}/\Gamma)\leq\frac{15}{256}.\]
The lower and upper bounds are sharp.
\end{thm}
\begin{proof}
Since $f(z)\in\mathcal{S}_{S,e}^*$  then
$\frac{2zf'(z)}{f(z)-f(-z)}= e^{w(z)}$ where $w(z)$ is a Schwarz function defined on   $\mathbb{D}$
or equivalently,
\[\frac{2zf'(z)}{f(z)-f(-z)}=e^\frac{p(z)-1}{p(z)+1}=1+\frac{1}{2}{p_1} z+\frac{1}{8}  (4 {p_2}-{p_1}^2)z^2+\frac{1}{48}  ({p_1}^3-12 {p_1} {p_2}+24 {p_3})z^3+\cdots\]
 where $p\in\mathcal{P}$.
On comparing the coefficients of like powers of $z$ on the both sides, we have 
\begin{align}
a_2=&\frac{1}{4}p_1\label{eq1},\\
a_3=&\frac{1}{16}(-p_1^2+4p_2)\label{eq1a},\\
a_4=&\frac{1}{384}(-p_1^3-12p_1p_2+48p_3)\label{eq1b},\\
a_5=&\frac{1}{384}(p_1^4-24p_1^2p_3+48p_4)\label{eq1c}.
\end{align}
On  substituting the value of $a_2$ and $a_3$  from equation \eqref{eq1} and \eqref{eq1a} in \eqref{eq105}, we have
\begin{align*}
T_{2,1}(	F_{f}/\gamma)=\frac{1}{16}\bigg(-\frac{p_1^4}{256}+\frac{p_1^2}{4}+\frac{p_1^2}{64}\RE(-p_1^2+4p_2)-\frac{1}{64}|-p_1^2+4p_2|^2\bigg).
\end{align*}
On  applying  Lemma \eqref{lem1}, we have
\begin{equation}\label{eq4}
T_{2,1}((	F_{f}/\gamma))=\frac{1}{4096}(-p_1^4+64p_1^2-8p_1^2(4-p_1^2)\Re(\xi)-16(4-p_1^2)^2|\xi|^2).
\end{equation}
Using the fact that $-\Re(\xi)\leq|\xi|,$  setting $x=|\xi|\in[0,1].$ Throughout  this section without loss of generality assume  $p=p_1\in[0,2].$ Thus equation  \eqref{eq4} becomes
\begin{align}\label{eq45}
T_{2,1}(	F_{f}/\gamma)&\leq\frac{1}{4096}(-p^4+64p^2+8p^2(4-p^2)x-16(4-p^2)^2x^2)\notag\\&=\Phi(p,x).
\end{align}
Next, we will obtain  the maximum value obtained by the function $\Phi$ on the rectangular shaped region $\Delta=[0,2]\times[0,1].$ We consider the following  two cases as follows.
\begin{itemize}
\item[ $(A_1)$] On  considering  boundary  of $\Delta,$ we have 
\begin{align*}
 &\Phi(0,x)=\frac{-x^2}{16}\leq0,\quad\Phi(2,x)=\frac{15}{256},\quad\Phi(p,0)=\frac{1}{4096}(-p^4+16p^2)\leq\frac{15}{256},\\&\Phi(p,1)=\frac{1}{4096}(-25p^4+224p^2-256)\leq\frac{15}{256}.
\end{align*}
\item[ $(A_2)$]Next, we consider interior part of $\Delta.$  For a critical point it is necessary that  the system of equations ${\partial \Phi(p,x)}/{\partial p}=0$ and ${\partial \Phi(p,x)}/{\partial x}=0$ has solution. A simple calculation provides that the equation  ${\partial \Phi(p,x)}/{\partial x}=0$  gives a solution $x=p^2/(4(4-p^2))=x_p\in(0,1)$ for $p<\frac{4\sqrt{5}}{5}\approx1.788\in(0,2).$ Further on solving   ${\partial \Phi(p,x)}/{\partial p}=0$ and substituting the value $x_p$, we get $p^5-8p^3+16p=0$ which is not possible for $p\in(0,2).$
Thus, the function $\Phi$ has no maximum value in the interior of $\Delta.$
\end{itemize}
Therefore,  we  get the upper bound ${15}/{256}$ on $T_{2,1}(	F_{f}/\gamma)$ and  sharpness follows by the function $f_2$ defined by \eqref{eqt2}.
Next, we proceed for the minimum value of  $T_{2,1}(f).$ Using  identity $-\Re(\xi)\geq-|\xi|$ and on setting $ x=|\xi|\in[0,1]$ in  \eqref{eq4},  we get
\begin{align}\label{eq46}
T_{2,1}(F_{f}/\gamma)&\geq\frac{1}{4096}(-p_1^4+64p_1^2-8p_1^2(4-p_1^2)x-16(4-p_1^2)^2x^2)=\Psi(p,x).
\end{align}
For the minimum value  of the function $\Psi(p,x)$ on the region $\Delta=[0,2]\times[0,1],$ following two cases arises.
\begin{itemize}
\item [$(B_1)$] First, we consider boundary of $\Delta.$ A simple calculations give following observations:
\begin{align*}
 &\Psi(0,x)=-\frac{x^2}{16}\geq-\frac{1}{16},\quad\Psi(2,x)=\frac{15}{256},\quad\Psi(p,0)=\frac{1}{4096}(-p^4+16p^2)\geq0,
 \\&\Psi(p,1)=\frac{1}{4096}(-9p^4+160p^2-256)\geq-\frac{1}{16}.
\end{align*}
\item [$(B_2)$]  Let $(p,x)\in(0,2)\times(0,1).$ Then   ${\partial \Psi(p,x)}/{\partial x}={-((4-p^2)(p^2+4x(4-p^2)))}/512\neq0.$ Hence, the system of equations ${\partial \Psi}/{\partial p}=0$ and ${\partial \Psi}/{\partial x}=0$  has no solution. The function $\Psi$ has no maximum value. 
\end{itemize}
In view of  $(B_1)$ and $(B_2)$,  the lower bound on $T_{2,1}(F_{f}/\gamma)$ is ${-1}/{16}$ which  is sharp for function $f_1$ defined by \eqref{eqt1}.

Next, in view of  \eqref{eq11}, \eqref{eq12}, \eqref{eq107}, \eqref{eq1} and \eqref{eq1a}, we have
\begin{align*}
T_{2,1}(F_{f^{-1}}/\Gamma)=\frac{1}{4096}(-25p^4+64p^2+80p^2\Re{p_2}-64|p_2|^2).
\end{align*}
On applying Lemma \eqref{lem1},  above expression becomes 
\[T_{2,1}(F_{f^{-1}}/\Gamma)=\frac{1}{4096}(-p^4+64p^2+8p^2(4-p^2)\Re (\xi)-16(4-p^2)^2|\xi|^2).\]
Using the fact $\Re(\xi)\leq|\xi|$  and on setting $x=|\xi|\in[0,1],$ we get 
$T_{2,1}(F_{f^{-1}}/\Gamma)\leq\Phi(p,x),$ which is  given by \eqref{eq45}. Hence $T_{2,1}(F_{f^{-1}}/\Gamma)\leq {15}/{256}.$
Again,  using $\Re(\xi)\geq-|\xi|$ and on setting $x=|\xi|,$
we have $T_{2,1}(F_{f^{-1}}/\Gamma)\geq\Psi(p,x),$ which is given by \eqref{eq46}. Hence $T_{2,1}(F_{f^{-1}}/\Gamma)\geq -{1}/{16}.$
The upper and lower bound  are sharp for the functions defined by \eqref{eqt2} and \eqref{eqt1} respectively.\qedhere
\end{proof}
Next two results prove that the bounds on $T_{2,1}(	F_{f}/\gamma)$ and $T_{2,1}(F_{f^{-1}}/\Gamma)$ for the function  $f\in\mathcal{S}_{S,L}^*$ are not invariant.
\begin{thm}
If the function $f\in \mathcal{S}_{S,	L}^*,$ then 
$$-\frac{1}{64}\leq T_{2,1}(	F_{f}/\gamma)\leq\frac{55}{4096}.$$
The bounds are sharp.
\end{thm}
\begin{proof}
Since $f\in \mathcal{S}_{S,	L}^*,$ on  substituting the value of $a_2$ and $a_3$  from \eqref{eq52} and \eqref{eq53} in  \eqref{eq105},  we get
\begin{equation}\label{eq70}
T_{2,1}(	F_{f}/\gamma)=\frac{1}{16}\bigg(-\frac{p_1^4}{4096}+\frac{p_1^2}{16}-\frac{5p_1^4}{1024}+\frac{p_1^2}{128}\Re(p_2)-\frac{1}{1024}|-5p_1^2+8p_2|^2\bigg).
\end{equation}
By making use Lemma \eqref{lem1} in  \eqref{eq70}, we get
\begin{equation}\label{eq71}
 T_{2,1}(	F_{f}/\gamma)=\frac{1}{65536}(-9p^4+256p^2+48p^2(4-p^2)\Re(\xi)-64(4-p^2)^2|\xi|^2).
\end{equation}
Using  $\Re(\xi)\leq|\xi|,$ the equation \eqref{eq71} becomes
\begin{align*}
 T_{2,1}(	F_{f}/\gamma)&\leq\frac{1}{65536}(-9p^4+256p^2+48p^2(4-p^2)|\xi|-64(4-p^2)^2|\xi|^2)\notag\\&=\kappa(p,|\xi|).
\end{align*}
On setting $|\xi|=x\in[0,1]$ and for  $p\in[0,2]$, we get $\kappa(p,x)=(-9p^4+256p^2+48p^2(4-p^2)x-64(4-p^2)^2x^2)/65536.$ To find the maximum value $\kappa$ over  $\Delta=[0,2]\times[0,1]$, we consider the following two cases.
\begin{itemize}
\item[ $(A_1)$] On the boundary of $\Delta$, we have 
\begin{align*}
 &\kappa(0,x)=\frac{-x^2}{64}\leq0,\quad\kappa(2,x)=\frac{55}{4096},\quad\kappa(p,0)=\frac{1}{65536}(-9p^4+256p^2)\leq\frac{55}{4096},\\&\kappa(p,1)=\frac{1}{65536}(-121p^4+960p^2-1024)\leq\frac{55}{4096}.
\end{align*}
\item[ $(A_2)$] In the interior of $\Delta$, for the existence of  critical point,    the system of equations ${\partial \kappa(p,x)}/{\partial p}=0$ and ${\partial \kappa(p,x)}/{\partial x}=0$ should have a solution. The equation  ${\partial \kappa(p,x)}/{\partial x}=0$  gives a solution $x=3p^2/(8(4-p^2))=x_p\in(0,1)$ when $p<{\sqrt{\frac{32}{11}}}\in(0,2).$ Further, on solving   ${\partial \kappa(p,x)}/{\partial p}=0$ and substituting the value $x_p$, we get $2p^5-16p^3+32p=0$ which is not possible for $p\in(0,2).$
The function $\kappa$ has no maximum value in the interior  of $\Delta.$
\end{itemize}
In view of cases $(A_1)$ and $(A_2)$ we  get the desired upper bound ${55}/{4096}$ on  $T_{2,1}(F_{f}/\gamma)$. This upper bound is sharp for  the function $g_2$ defined  by \eqref{eqt72}.

Next, by using  $\Re(\xi)\geq-|\xi|$ and  on setting   $x=|\xi|\in[0,1]$ in  \eqref{eq71},  we get
\begin{align*}
 T_{2,1}(	F_{f}/\gamma)&\geq\frac{1}{65536}(-9p^4+256p^2-48p^2(4-p^2)x-64(4-p^2)^2x^2)=G(p,x).
\end{align*}
For $(p,x)\in(0,2)\times(0,1)$, note that ${\partial G(p,x)}/{\partial x}={-((4-p^2)(3p^2+8x(4-p^2)))}/4096\neq 0$. The system of equations ${\partial G}/{\partial p}=0$ and ${\partial G}/{\partial x}=0$  has no common  solution. Therefore, the function $G$ has no minimum value in the interior of $\Delta$. On the boundary of $\Delta,$ a simple calculation yields
\begin{align*}
 &G(0,x)=-\frac{x^2}{64}\geq-\frac{1}{64},\quad G(2,x)=\frac{55}{4096},\quad G(p,0)=\frac{1}{65536}(-9p^4+256p^2)\geq0,\\&G(p,1)=\frac{1}{4096}(-25p^4+576p^2-1024)\geq-\frac{1}{64}.
\end{align*}
Therefore, the  lower bound on $T_{2,1}(	F_{f}/\gamma)$ is $-{1}/{64}.$ Sharpness follows for the function $g_1$ defined by \eqref{eqt7}.
\end{proof}

\begin{thm}
Let the function $f\in \mathcal{S}_{S, L}^*$. Then 
\[-\frac{1}{64}\leq T_{2,1}(	F_{f^{-1}}/\Gamma)\leq\frac{39}{4096}.\]
The  inequality is sharp.
\end{thm}
\begin{proof}
Since $f\in \mathcal{S}_{S,L}^*$, in view of  \eqref{eq107}, \eqref{eq1} and \eqref{eq1a}, we get
\begin{equation}\label{eq74}
T_{2,1}(	F_{f}/\gamma)=\frac{1}{16}\bigg(-\frac{p_1^4}{4096}+\frac{p_1^2}{16}+\frac{7p_1^4}{1024}-\frac{p_1^2}{128}\Re(p_2)-\frac{1}{1024}|7p_1^2-8p_2|^2\bigg).
\end{equation}
Using Lemma \eqref{lem1}, the expression \eqref{eq74}  becomes
\begin{equation}\label{eq75}
 T_{2,1}(	F_{f}/\gamma)=\frac{1}{65536}(-25p^4+256p^2+80p^2(4-p^2)\Re(\xi)-64(4-p^2)^2|\xi|^2).
\end{equation}
Using $\Re(\xi)\leq|\xi|,$ the above expression becomes
\begin{align*}
 T_{2,1}(	F_{f}/\gamma)&\leq\frac{1}{65536}(-25p^4+256p^2+80p^2(4-p^2)|\xi|-64(4-p^2)^2|\xi|^2)=\delta(p,|\xi|).
\end{align*}
On setting $|\xi|=x\in[0,1]$, we have  $\delta(p,x)=(-25p^4+256p^2+80p^2(4-p^2)x-64(4-p^2)^2x^2)/65536$ for $p\in[0,2]$.  For the maximum value of the function  $\delta$ on $\Delta=[0,2]\times[0,1],$ we consider two cases:
\begin{itemize}
\item[$(A_1)$] On the boundary of $\Delta,$ we note that
\begin{align*}
 &\delta(0,x)=-\frac{x^2}{64}\leq0,\quad\delta(2,x)=\frac{39}{4096},\quad\delta(p,0)=\frac{1}{65536}(-9p^4+256p^2)\leq\frac{39}{4096},\\&\delta(p,1)=\frac{1}{65536}(-169p^4+1088p^2-1024)\leq\frac{39}{4096}.
\end{align*}
\item[ $(A_2)$] In the interior of $\Delta$, the equation  ${\partial \delta(p,x)}/{\partial x}=0$   gives a solution $x=5p^2/(8(4-p^2))=x_p\in(0,1),$ when $p<{\sqrt{\frac{32}{13}}}\in(0,2).$ Further,  we substitute the value $x_p$ in  ${\partial \delta(p,x)}/{\partial p}=0$  gives $2p^5-16p^3+32p=0,$ which is not possible for $p\in(0,2).$
The function $\delta$ has no maximum value in the interior  of $\Delta.$
\end{itemize}
Hence from $(A_1)$ and $(A_2),$ the  upper bound on $T_{2,1}(	F_{f}/\gamma)$ is ${39}/{4096}.$ Sharpness follows for  the function $g_2$ defined by \eqref{eqt72}.

On using  the identity $\Re(\xi)\geq-|\xi|$ and on setting    $x=|\xi|\in[0,1]$ in equation \eqref{eq75},  we have
\begin{align*}
 T_{2,1}(	F_{f}/\gamma)&\geq\frac{1}{65536}(-25p^4+256p^2-80p^2(4-p^2)x-64(4-p^2)^2x^2)=J(p,x).
\end{align*}
Let $(p,x)\in\Delta$ and it is noted that ${\partial J(p,x)}/{\partial x}={-((4-p^2)(5p^2+8x(4-p^2)))}/4096\neq0.$ The system of equation ${\partial J}/{\partial p}=0$ and ${\partial J}/{\partial x}=0$  has no solution. The function $J$ has no minimum value in the interior of $\Delta$. Hence the minimum value of the function $J$ attained on the boundary of $\Delta$ such that
\begin{align*}
 &J(0,x)=-\frac{x^2}{64}\geq-\frac{1}{64},\quad J(2,x)=\frac{39}{4096},\quad J(p,0)=\frac{1}{65536}(-9p^4+256p^2)\geq0,\\&J(p,1)=\frac{1}{4096}(-9p^4+448p^2-1024)\geq-\frac{1}{64}.
\end{align*}
 Thus, the lower bound on $T_{2,1}(	F_{f}/\gamma)$ is  $-{1}/{64}$ which is sharp for the function $g_1$ defined by \eqref{eqt7}. \qedhere
\end{proof}
\section*{Acknowledgement}
The first and the corresponding author would like to thank the Institute of Eminence, University of Delhi, Delhi, India--110007 for providing financial support for this research under grant number /IoE/2021/12/FRP. The second author would like to thank  the UGC Non-NET Fellowship for supporting financially vide Ref. No. Sch/139/Non-NET/Ext-156/2022-2023/722.

\end{document}